\documentclass[12pt]{article}

\usepackage[margin=1in]{geometry}
\usepackage{amsthm}
\usepackage{amssymb}
\usepackage{amsmath}
\usepackage{graphicx}
\usepackage{mathdots}
\usepackage{mathtools}
\usepackage{url}
\usepackage{color}
\usepackage{framed}

\newtheorem{theorem}{Theorem}%[section]

\newtheorem{proposition}[theorem]{Proposition}

\theoremstyle{definition}

\newtheorem{example}[theorem]{Example}
\newtheorem{definition}[theorem]{Definition}

\title{Equi-isoclinic subspaces, covers of the complete graph, and complex conference matrices}

\author{Matthew~Fickus\footnote{Department of Mathematics \& Statistics, Air Force Institute of Technology, Wright-Patterson AFB, OH} \quad Joseph~W.~Iverson\footnote{Department of Mathematics, Iowa State University, Ames, IA} \quad John~Jasper\footnotemark[1] \quad Dustin~G.~Mixon\footnote{Department of Mathematics, The Ohio State University, Columbus, OH} \footnote{Translational Data Analytics Institute, The Ohio State University, Columbus, OH}}
\date{}

\begin{document}
\maketitle

\begin{abstract}
In 1992, Godsil and Hensel published a ground-breaking study of distance-regular antipodal covers of the complete graph that, among other things, introduced an important connection with equi-isoclinic subspaces.
This connection seems to have been overlooked, as many of its immediate consequences have never been detailed in the literature.
To correct this situation, we first describe how Godsil and Hensel's machine uses representation theory to construct equi-isoclinic tight fusion frames.
Applying this machine to Mathon's construction produces $q+1$ planes in $\mathbb{R}^{q+1}$ for any even prime power $q>2$. 
Despite being an application of the $30$-year-old Godsil--Hensel result, infinitely many of these parameters have never been enunciated in the literature.
Following ideas from Et-Taoui, we then investigate a fruitful interplay with complex symmetric conference matrices.
\end{abstract}

Consider the problem of arranging a fixed number of points in the Grassmannian so that the minimum pairwise distance is maximized.
There are a few choices of ``distance'' that one could use here, but for several choices, a sufficient optimality condition is that the subspaces form what is known as an \textit{equi-isoclinic tight fusion frame} (or \textsc{eitff}); see~\cite{DhillonHST:08,FickusJMW:17,FickusIJM:22}.
Given matrices $M_1,\ldots,M_n\in\mathbb{C}^{d\times r}$, each with orthonormal columns, we say the subspaces $\{\operatorname{im}(M_i)\}_{i=1}^n$ form an \textsc{eitff} if there exist $\alpha,\beta>0$ such that
\[
\alpha M_i^*M_j\in\operatorname{U}(r) \text{ whenever } i\neq j
\qquad
\text{ and }
\qquad
\sum_{i=1}^n M_iM_i^*=\beta I_d,
\]
where $\operatorname{U}(r)$ denotes the group of $r\times r$ unitary matrices and $I_d$ denotes the $d\times d$ identity matrix.
If $M_1,\ldots,M_n\in\mathbb{R}^{d\times r}\subseteq\mathbb{C}^{d\times r}$, then their images in $\mathbb{R}^d$ create a \textit{real} \textsc{eitff}.
The special case where $r=1$ reduces to a well-studied object known as an \textit{equiangular tight frame}~\cite{StrohmerH:03,FickusM:15}.
(While mutually orthogonal subspaces with direct sum $\mathbb{C}^d$ are typically considered to be examples of \textsc{eitff}s, our formulation does not include them for convenience.)

In this paper, we interact with \textsc{eitff}s using a different type of matrix representation.
Denoting $M:=[M_1\cdots M_n]\in(\mathbb{C}^{d\times r})^{1\times n}$, we refer to $M^*M\in(\mathbb{C}^{r\times r})^{n\times n}$ as the corresponding \textit{fusion Gram matrix}.
Observe that a matrix $G\in(\mathbb{C}^{r\times r})^{n\times n}$ is an \textsc{eitff} fusion Gram matrix if it satisfies the following properties for some choice of scalars $\alpha,\beta>0$:
\begin{itemize}
\item
$G_{ii}=I_r$ for every $i\in[n]$,
\item
$\alpha G_{ij}\in\operatorname{U}(r)$ for every $i,j\in[n]$ with $i\neq j$, and
\item
$G$ is $\beta$ times an orthogonal projection matrix.
\end{itemize}
Furthermore, an \textsc{eitff} fusion Gram matrix determines the $\operatorname{U}(d)$ orbit of its underlying \textsc{eitff}.
(To identify a member of this orbit, one may simply factor $G=M^*M$.)
Subtracting the identity from an \textsc{eitff} fusion Gram matrix and then dividing by $\alpha$ results in the primary object of interest in this paper:

\begin{definition}
We say $S\in(\mathbb{C}^{r\times r})^{n\times n}$ is an \textsc{eitff} \textbf{signature matrix} if
\begin{itemize}
\item[(S1)]
$S_{ii}=0$ for every $i\in[n]$,
\item[(S2)]
$S_{ij}\in\operatorname{U}(r)$ for every $i,j\in[n]$ with $i\neq j$,
\item[(S3)]
$(S_{ij})^*=S_{ji}$ for every $i,j\in[n]$ with $i\neq j$, and
\item[(S4)]
the minimal polynomial of $S$ has degree $2$.
\end{itemize}
\end{definition}

Observe that (S1) implies $\operatorname{tr}(S)=0$, (S2) implies $S$ is nonzero, (S3) implies the eigenvalues of $S$ are all real, and (S4) implies $S$ has at most two eigenvalues.
It follows that $S$ has exactly one positive eigenvalue and exactly one negative eigenvalue.
We call $S$ a $(d,n,r)$-\textsc{eitff} signature matrix if its positive eigenvalue has multiplicity $d$.

Given an \textsc{eitff} signature matrix $S$, the corresponding \textsc{eitff} fusion Gram matrix (and hence the \textsc{eitff} itself) can be recovered as $I_{rn}-\lambda_{\min}(S)^{-1}\cdot S$.
In this paper, we discuss how to construct \textsc{eitff} signature matrices using different combinatorial objects, namely, covers of the complete graph and complex conference matrices.

Given integers $n,m,c>0$, an $(n,m,c)$-\textsc{drackn} (short for \textit{distance-regular antipodal cover of the complete graph}) is a simple graph on a set $V$ of $mn$ vertices, together with a partition $\{V_i\}_{i=1}^n$ of $V$ into size-$m$ subsets (called \textit{fibers}) such that
\begin{itemize}
\item
vertices in a common fiber are nonadjacent,
\item
each pair of distinct fibers induces a perfect matching between those fibers, and
\item
nonadjacent vertices from distinct fibers have exactly $c$ common neighbors.
\end{itemize}
(While the \textsc{drackn} literature typically uses $r$ in place of our $m$, we use $m$ to disambiguate from the \textsc{eitff} parameter.)
We note that the partition into fibers can be recovered from the graph data as equivalence classes of \textit{antipodal} vertices (in this case, vertices of pairwise distance $3$).
The theory of \textsc{drackn}s was explored in~\cite{GodsilH:92}, and has since been used to construct \textsc{eitff}s with $r=1$~\cite{CoutinhoCSZ:16}.
We are interested in adjacency matrices of such graphs, and it is convenient to use the fibers to define a block structure for such a matrix:

\begin{definition}
We say $A\in(\mathbb{C}^{m\times m})^{n\times n}$ is an $(n,m,c)$-\textsc{drackn} \textbf{adjacency matrix} if
\begin{itemize}
\item[(D1)]
$A_{ii}=0$ for every $i\in[n]$,
\item[(D2)]
$A_{ij}$ is a permutation matrix for every $i,j\in[n]$ with $i\neq j$,
\item[(D3)]
$(A_{ij})^\top=A_{ji}$ for every $i,j\in[n]$ with $i\neq j$, and
\item[(D4)]
$A^2=(n-mc-2)A+(n-1)I_{mn}+c(J_n-I_n)\otimes J_m$,
\end{itemize}
where $J_n$ denotes the $n\times n$ all-ones matrix.
\end{definition}

Indeed, (D1) captures the nonadjacency within fibers, (D2) captures the induced perfect matchings between fibers, (D3) says that the graph is not directed, and (D4) expresses the definition of $c$ in terms of the adjacency matrix.
An $(n,m,c)$-\textsc{drackn} adjacency matrix has eigenvalues $n-1$ and $-1$ with respective multiplicities $1$ and $n-1$.
There are two other eigenvalues denoted by $\theta$ and $\tau$, which we express in terms of $\delta:=n-mc-2$:
\[
\theta:=\frac{\delta+\sqrt{\delta^2+4(n-1)}}{2},
\qquad
\tau:=\frac{\delta-\sqrt{\delta^2+4(n-1)}}{2}.
\]
The off-diagonal blocks generate a group of permutation matrices that we denote by $\Gamma$.
For any representation $\rho\colon\Gamma\to\operatorname{U}(d_\rho)$, we take $\hat\rho(A)\in(\mathbb{C}^{d_\rho\times d_\rho})^{n\times n}$ to be the matrix with zero blocks on the diagonal and off-diagonal blocks $(\hat\rho(A))_{ij}=\rho(A_{ij})$.

\begin{theorem}[cf.\ \cite{GodsilH:92}]
\label{thm.main result}
Given an $(n,m,c)$-\textsc{drackn} adjacency matrix $A$ with eigenvalues $n-1$, $\theta$, $-1$, and $\tau$, let $\Gamma$ denote the group of permutation matrices generated by the off-diagonal blocks of $A$.
Select any degree-$r$ constituent representation $\pi$ of the inclusion map $\Gamma\hookrightarrow\operatorname{U}(m)$ such that $\pi$ does not contain the trivial representation as a constituent.
Then $\hat\pi(A)$ is a signature matrix of a $(d,n,r)$-\textsc{eitff} with $d=\frac{rn|\tau|}{\theta-\tau}$.
Furthermore, this \textsc{eitff} is real if $\pi$ is real.
\end{theorem}

We emphasize that $\pi$ may be a reducible representation; see Example~\ref{ex.godsil hensel}, for instance.
Furthermore, the requirement that $\pi$ does not contain the trivial representation requires the \textsc{drackn} to have fibers of size $m>1$, i.e., it is not the complete graph of order $n$.
Interestingly, the redundancy $\frac{rn}{d}=\frac{\theta-\tau}{|\tau|}$ of the \textsc{eitff} from Theorem~\ref{thm.main result} is determined by the \textsc{drackn} parameters $(n,m,c)$, i.e., it does not depend on the choice of $\pi$.
In the case where $\pi$ is irreducible, Theorem~\ref{thm.main result} is described in the final paragraph of Section~13 in~\cite{GodsilH:92}.
We will use this result to produce never-before-seen parameters of equi-isoclinic tight fusion frames.
For completeness, we flesh out the proof of Theorem~\ref{thm.main result} below.

\begin{proof}[Proof of Theorem~\ref{thm.main result}]
First, (S1)--(S3) follow immediately from (D1)--(D3) and the definition of $\hat\pi$.
The characteristic polynomial of $A$ takes the form
\begin{equation}
\label{eq.char poly 1}
\chi_A(x)
=(x-n+1)(x-\theta)^{m_\theta}(x+1)^{n-1}(x-\tau)^{m_\tau}
\end{equation}
for some $m_\tau,m_\theta\in\mathbb{N}$.
In what follows, we use basic ideas from representation theory to show that the characteristic polynomial of $\hat\pi(A)$ divides $(x-\theta)^{m_\theta}(x-\tau)^{m_\tau}$.
We decompose the inclusion map $\iota\colon\Gamma\hookrightarrow\operatorname{U}(m)$ and the constituent $\pi\colon\Gamma\to\operatorname{U}(r)$ as direct sums of irreducibles:
\[
\iota\cong\bigoplus_{j\in[k]}\pi_j,
\qquad
\pi\cong\bigoplus_{j\in S}\pi_j,
\]
where $S\subseteq[k]$.
By the matrix similarities
\[
\hat\iota(A)\cong\bigoplus_{j\in[k]}\hat\pi_j(A),
\qquad
\hat\pi(A)\cong\bigoplus_{j\in S}\hat\pi_j(A),
\]
we may factor the characteristic polynomials:
\[
\chi_{\hat\iota(A)}(x)=\prod_{j\in[k]}\chi_{\hat\pi_j(A)}(x),
\qquad
\chi_{\hat\pi(A)}(x)=\prod_{j\in S}\chi_{\hat\pi_j(A)}(x).
\]
Since $\iota$ is a permutation representation, the trivial representation is a constituent, which we take to be $\pi_1$ without loss of generality.
Then $\hat\pi_1(A)=J_n-I_n$, and so
\[
\chi_{\hat\pi_1(A)}(x)=(x-n+1)(x+1)^{n-1}.
\]
By our assumption on $\pi$, we have $1\not\in S$, which in turn gives
\begin{align}
\label{eq.char poly 2}
\chi_A(x)
&=\chi_{\hat\iota(A)}(x)
=\prod_{j=1}^k\chi_{\hat\pi_j(A)}(x)
\nonumber
=\prod_{j\in S}\chi_{\hat\pi_j(A)}(x)\cdot\prod_{j\not\in S}\chi_{\hat\pi_j(A)}(x)\\
&=\chi_{\hat\pi(A)}(x)\cdot(x-n+1)(x+1)^{n-1}\cdot\prod_{j\not\in S\cup\{1\}}\chi_{\hat\pi_j(A)}(x).
\end{align}
Thus, combining \eqref{eq.char poly 1} and \eqref{eq.char poly 2} gives
\[
\chi_{\hat\pi(A)}(x)
=(x-\theta)^{a}(x-\tau)^{b}
\]
for some $a,b\in\mathbb{N}\cup\{0\}$.
Furthermore, $\operatorname{tr}\hat\pi(A)=0$ implies $a,b>0$.
This gives (S4).
To determine $d$, we observe that $a$ and $b$ satisfy $a+b=rn$ and $a\theta+b\tau=0$, and so
\[
a=-\frac{rn\tau}{\theta-\tau},
\qquad
b=\frac{rn\theta}{\theta-\tau}.
\]
Since $\theta>0$ is the positive eigenvalue of $\hat\pi(A)$, it follows that $d=a$.
Finally, when $\pi$ is real, the resulting fusion Gram matrix is real, too.
\end{proof}

In what follows, we highlight a few applications of Theorem~\ref{thm.main result}.

\begin{example}[Abelian \textsc{drackn}s and equiangular lines]
\label{ex.abelian drackns}
A \textsc{drackn} is said to be \textit{abelian} if the group $\Gamma\leq S_m$ generated by the off-diagonal blocks of the \textsc{drackn} adjacency matrix is abelian and acts regularly on $[m]$.
For this class of \textsc{drackn}s (and restricting to irreducible representations), Theorem~\ref{thm.main result} amounts to Theorem~4.1 in~\cite{CoutinhoCSZ:16}, and most of the immediate consequences of Theorem~\ref{thm.main result} are summarized in \cite{CoutinhoCSZ:16}.
Here, the irreducible constituent representations all have degree~$1$, and so the resulting \textsc{eitff}s consist of equiangular lines.
\end{example}

\begin{example}[Section~11 in~\cite{GodsilH:92}]
\label{ex.godsil hensel}
Given an $(n,m,c)$-\textsc{drackn} adjacency matrix $A$, Godsil and Hensel~\cite{GodsilH:92} consider a matrix $M\in(\mathbb{R}^{m_\theta\times m})^{1\times n}$ whose rows form an orthonormal basis for the $\theta$-eigenspace of $A$.
The images of the $m_\theta\times m$ blocks of $M$ form a real ${(m_\theta,n,m-1)}$-\textsc{eitff}.
An equivalent construction arises from Theorem~\ref{thm.main result} by taking $\pi$ to be the degree-$(m-1)$ representation given by the action of $\Gamma$ on the orthogonal complement of the all-ones vector in $\mathbb{R}^m$.
For example, a construction due to Gardiner~\cite{Gardiner:74} gives a $(7,6,1)$-\textsc{drackn}, which produces a real $(21,7,5)$-\textsc{eitff}.
\end{example}

\begin{example}[Mathon~\textsc{drackn}s]
\label{ex.mathon}
Next, we describe \textsc{drackn}s (due to Mathon~\cite{Mathon:75,BrouwerCN:87,GodsilH:92}) for which Theorem~\ref{thm.main result} produces \textsc{eitff}s beyond those covered by Examples~\ref{ex.abelian drackns} and~\ref{ex.godsil hensel}.
Fix an even prime power $q$.
Take $V:=\mathbb{F}_q^2$ with a symplectic form $B$.
Define $G$ to be the graph with vertex set $V\setminus\{0\}$ with $u$ adjacent to $v$ precisely when $B(u,v)=1$.
Then $G$ is a $(q+1,q-1,1)$-\textsc{drackn}.
As we discuss in the proof of Theorem~\ref{thm.mathon eitffs}, the resulting permutation group is isomorphic to the dihedral group of order $2(q-1)$, and when $q>2$, Theorem~\ref{thm.main result} produces an \textsc{eitff} consisting of $q+1$ planes in $\mathbb{R}^{q+1}$.
\end{example}

The \textsc{eitff} parameters enunciated in Example~\ref{ex.mathon} are new to the literature for infinitely many $q$, but they could have been obtained $30$ years ago by combining ideas from Godsil and Hensel~\cite{GodsilH:92} with Mathon's \textsc{drackn}.
Interestingly, there is now a \textit{second} combination of results from the more recent literature that produces these same \textsc{eitff} parameters.
In particular, one may combine the \textit{complex symmetric conference matrices} given in Theorem~5.2 of~\cite{FickusIJM:21} with a key result due to Et-Taoui~\cite{EtTaoui:18}, namely, Proposition~\ref{prop.et-taoui} below.
This reveals a fruitful interplay between \textsc{drackn}s, \textsc{eitff}s, and conference matrices that we investigate for the remainder of this paper.

\begin{definition}
We say $C\in\mathbb{C}^{n\times n}$ is a \textbf{symmetric conference matrix} if
\begin{itemize}
\item[(C1)]
$C_{ii}=0$ for every $i\in[n]$,
\item[(C2)]
$|C_{ij}|=1$ for every $i,j\in[n]$ with $i\neq j$,
\item[(C3)]
$C^\top=C$, and 
\item[(C4)]
$CC^*=(n-1)I_n$.
\end{itemize}
\end{definition}

\begin{proposition}[Et-Taoui~\cite{EtTaoui:18,BlokhuisBE:18}]
\label{prop.et-taoui}
Given $C\in\mathbb{C}^{n\times n}$, define $S\in(\mathbb{R}^{2\times 2})^{n\times n}$ by
\[
S_{ij}:=\left[\begin{array}{rr}
\operatorname{Re}(C_{ij})&\operatorname{Im}(C_{ij})\\
\operatorname{Im}(C_{ij})&-\operatorname{Re}(C_{ij})
\end{array}\right].
\]
Then $S$ is a $(d,n,r)$-\textsc{eitff} signature matrix with $d=n$ and $r=2$ if and only if $C$ is a symmetric conference matrix.
\end{proposition}

Recall the usual homomorphism of real $*$-algebras $\operatorname{ctr}\colon\mathbb{C}^{n\times n}\to(\mathbb{R}^{2\times 2})^{n\times n}$ that replaces each entry $a+bi$ with the corresponding block $[\begin{smallmatrix}a&-b\\b&\phantom{-}a\end{smallmatrix}]$.
This is related to (but different from) the replacement $a+bi\mapsto[\begin{smallmatrix}a&\phantom{-}b\\b&-a\end{smallmatrix}]$ used in Proposition~\ref{prop.et-taoui}.

\begin{proof}[Proof of Proposition~\ref{prop.et-taoui}]
It is clear that (S1)$\Leftrightarrow$(C1), (S2)$\Leftrightarrow$(C2), and (S3)$\Leftrightarrow$(C3).
Next, we verify that (S4)$\Leftrightarrow$(C4).
For ($\Leftarrow$), we have $S=\operatorname{ctr}(C)(I_n\otimes[\begin{smallmatrix}1&\phantom{-}0\\0&-1\end{smallmatrix}])$, and so
\begin{align}
S^2
\nonumber
&=SS^\top
=\operatorname{ctr}(C)(I_n\otimes[\begin{smallmatrix}1&\phantom{-}0\\0&-1\end{smallmatrix}])(I_n\otimes[\begin{smallmatrix}1&\phantom{-}0\\0&-1\end{smallmatrix}])^\top\operatorname{ctr}(C)^\top\\
\label{eq.et-taoui}
&=\operatorname{ctr}(C)\operatorname{ctr}(C)^\top
=\operatorname{ctr}(C)\operatorname{ctr}(C^*)
=\operatorname{ctr}(CC^*)
=\operatorname{ctr}((n-1)I_n)
=(n-1)I_{2n}.
\end{align}
Thus, $S$ is an \textsc{eitff} signature matrix.
Furthermore, $d=n$ follows from the fact that the eigenvalues of $S$ are $\pm\sqrt{n-1}$, each with multiplicity $n$, while $r=2$ follows from the fact that $S$ has $2\times 2$ blocks.

For ($\Rightarrow$), the positive eigenvalue of $S$ has multiplicity $d=n$ by assumption.
It follows that the negative eigenvalue of $S$ has multiplicity $2n-d=n$.
Since $\operatorname{tr}(S)=0$, the eigenvalues of $S$ are $\pm\lambda$ for some $\lambda>0$.
The squared Frobenius norm of $S$ is $2n(n-1)$, which necessarily equals $2n\lambda^2$, and so the minimal polynomial of $S$ is $x^2-(n-1)$.
That is, $S^2=(n-1)I_{2n}$.
The result then follows by an argument similar to \eqref{eq.et-taoui}.
\end{proof}

The proof of the following result reveals the relationship between Theorem~\ref{thm.main result} and Proposition~\ref{prop.et-taoui} in the case of Example~\ref{ex.mathon}.

\begin{theorem}
\label{thm.mathon eitffs}
For each $k>1$, there exists a symmetric conference matrix in $\mathbb{C}^{(2^k+1)\times(2^k+1)}$.
\end{theorem}

Recently, the authors constructed symmetric conference matrices of the same size in~\cite{FickusIJM:21} using a very different approach.
The proof of Theorem~\ref{thm.mathon eitffs} uses the representation theory of the dihedral group.
Let $m$ be an odd positive integer, and define the dihedral group $D_{2m}$ to be the group of bijections $g\colon\mathbb{Z}_m\to\mathbb{Z}_m$ defined by $g(x):= \epsilon x+b$ with $\epsilon=\pm1$ and $b\in\mathbb{Z}_m$.
Consider the action of $D_{2m}$ on $\mathbb{Z}_m$.
Denoting $\omega:=e^{2\pi i/m}$ and the Fourier mode $f_k:=\sum_{j=0}^{m-1} \omega^{-jk}e_j\in\mathbb{C}^m$, then the $m\times m$ matrix representation of the action $D_{2m}\circlearrowleft\mathbb{Z}_m$ can be decomposed as a direct sum of irreducible representations acting on the following mutually orthogonal subspaces: $\operatorname{span}\{f_0\}$ and $\operatorname{span}\{f_k,f_{-k}\}$ for $0<k<m/2$.
The representation is trivial on $\operatorname{span}\{f_0\}$, and the representation on $\operatorname{span}\{f_k,f_{-k}\}$ can be expressed in terms of the orthonormal basis obtained by normalizing the real and imaginary parts of $f_k$:
\[
x+1
\mapsto
\left[\begin{array}{rr}
\cos(\frac{2\pi k}{m})&-\sin(\frac{2\pi k}{m})\\
\sin(\frac{2\pi k}{m})&\cos(\frac{2\pi k}{m})
\end{array}\right],
\qquad
-x
\mapsto
\left[\begin{array}{rr}
1&0\\
0&-1
\end{array}\right].
\]
This gives a real representation of $D_{2m}$.

\begin{proof}[Proof of Theorem~\ref{thm.mathon eitffs}]
To construct the desired symmetric conference matrix, we apply Theorem~\ref{thm.main result} to the $(q+1,q-1,1)$-Mathon \textsc{drackn} to obtain an \textsc{eitff} signature matrix of the form given in Proposition~\ref{prop.et-taoui}.

Let $V$ and $B$ be as in Example~\ref{ex.mathon}.
Partition $V\setminus\{0\}$ into the $q+1$ different $1$-dimensional subspaces of $V$ (less zero), each consisting of $m:=q-1$ points.
These $m$ vertices form a fiber in the Mathon \textsc{drackn}.
Select nonzero representatives $\{u_k\}_{k=1}^{q+1}$ of the $1$-dimensional subspaces of $V$, and let $\alpha$ be a generator of the cyclic group $\mathbb{F}_q^*$.
For each $k,\ell\in[q+1]$ with $k\neq\ell$, select $\gamma_{k\ell}\in\mathbb{Z}_{q-1}$ such that $B(u_k,u_\ell)=\alpha^{\gamma_{k\ell}}$.
We use $(k,i)\in[q+1]\times[q-1]$ to label the vertex $\alpha^{i}u_k\in V\setminus\{0\}$.
Considering $B(\alpha^{i}u_k,\alpha^{j}u_\ell)=\alpha^{i+j}B(u_k,u_\ell)=\alpha^{i+j+\gamma_{k\ell}}$ for $k\neq\ell$, the adjacency matrix $A\in(\mathbb{R}^{(q-1)\times(q-1)})^{(q+1)\times(q+1)}$ with this vertex labeling is given by
\[
(A_{k\ell})_{ij}
:=\left\{\begin{array}{cl}
1&\text{if } k\neq\ell \text{ and } i+j+\gamma_{k\ell}\equiv0\bmod q-1,\\
0&\text{otherwise}.
\end{array}\right.
\]
When $k\neq\ell$, the block $A_{k\ell}$ resides in the image of the $(q-1)$-dimensional permutation representation of the dihedral group $D_{2(q-1)}$.
Explicitly, $A_{k\ell}$ represents the group element $x\mapsto-x-\gamma_{k\ell}$, which is not a pure translation since the coefficient of $x$ is $-1$.
As such, when evaluating $A$ at any irreducible constituent representation of degree~$2$ (which in turn exists since $k>1$), there exists $a\in\mathbb{Z}$ such that the resulting matrix $S\in(\mathbb{R}^{2\times 2})^{(q+1)\times(q+1)}$ has $S_{kk}=0$ and
\[
S_{k\ell}:=\left[\begin{array}{rr}
\cos(\frac{2\pi a\gamma_{k\ell}}{q-1})&\sin(\frac{2\pi a\gamma_{k\ell}}{q-1})\\
\sin(\frac{2\pi a\gamma_{k\ell}}{q-1})&-\cos(\frac{2\pi a\gamma_{k\ell}}{q-1})
\end{array}\right]
\qquad
\text{for }k\neq\ell.
\]
By Theorem~\ref{thm.main result}, $S$ is necessarily an \textsc{eitff} signature matrix.
Defining $C\in\mathbb{C}^{(q+1)\times(q+1)}$ by
\begin{equation}
\label{eq.mathon conference}
C_{k\ell}:=\left\{\begin{array}{cl}
\exp(\frac{2\pi ia\gamma_{k\ell}}{q-1}) &\text{if }k\neq\ell,\\
0&\text{otherwise},
\end{array}\right.
\end{equation}
then $S$ can be obtained from $C$ using Proposition~\ref{prop.et-taoui}, which in turn implies that $C$ is a symmetric conference matrix.
\end{proof}

Note that given the conference matrix defined by \eqref{eq.mathon conference}, one can undo the construction described in the proof of Theorem~\ref{thm.mathon eitffs} to recover the underlying Mathon \textsc{drackn}.
In what follows, we use a similar procedure to create \textsc{drackn}s from complex symmetric conference matrices whose off-diagonal entries are all $p$th roots of unity for a common prime $p$.

\begin{theorem}
\label{thm.conf to drackn}
Suppose there is a symmetric conference matrix in $\mathbb{C}^{n\times n}$ whose off-diagonal entries are all $p$th roots of unity for some common prime $p$.
Then $p$ divides $n-2$, and there exists an $(n,m,c)$-\textsc{drackn} with $m=p$ and $c=(n-2)/p$.
\end{theorem}

\begin{proof}
For the first claim, we take inspiration from the proof of Theorem~5.1 in~\cite{CoutinhoCSZ:16}.
Let $C\in\mathbb{C}^{n\times n}$ denote the complex symmetric conference matrix.
Whenever $i\neq j$, then (C4) and (C1) together give
\begin{equation}
\label{eq.sum of pth roots of unity}
0
=(CC^*)_{ij}
=\sum_{k=1}^n C_{ik}\overline{C_{jk}}
=\sum_{\substack{k=1\\k\not\in\{i,j\}}}^n C_{ik}\overline{C_{jk}}.
\end{equation}
Let $\omega$ denote a primitive $p$th root of unity, and take $a_\ell$ to be the number of times $\omega^\ell$ appears as a term in the right-hand side of \eqref{eq.sum of pth roots of unity}.
Then $\sum_{\ell=0}^{p-1}a_\ell\omega^\ell=0$, and counting terms gives $\sum_{\ell=0}^{p-1}a_\ell=n-2$.
Since $\{\omega^\ell\}_{\ell=1}^{p-1}$ is linearly independent over $\mathbb{Q}$, the unique rational dependence between $\{\omega^\ell\}_{\ell=0}^{p-1}$ is given by $\sum_{\ell=0}^{p-1}\omega^\ell=0$ (up to scaling).
Thus, $a_\ell=(n-2)/p$ for every $\ell$, which in turn implies that $p$ divides $n-2$.

For the second claim, we take inspiration from Section~5 in~\cite{BlokhuisBE:18} to construct a \textsc{drackn} adjacency matrix.
Define the translation and reversal matrices $T,R\in\mathbb{C}^{p\times p}$ by indexing the rows and columns by $\mathbb{Z}/p\mathbb{Z}$ and putting $Te_i:=e_{i+1}$ and $Re_i:=e_{-i}$.
Next, whenever $i\neq j$, take $\ell(i,j)$ to be the unique member of $\{0,\ldots,p-1\}$ such that $C_{ij}=\omega^{\ell(i,j)}$.
Then we define $A\in(\mathbb{C}^{p\times p})^{n\times n}$ by $A_{ii}:=0$ and $A_{ij}:=T^{\ell(i,j)}R$ whenever $i\neq j$.
Then $A$ satisfies (D1) and (D2) immediately.
For (D3), we take the $\ell(i,j)$th power of $T^{-1}=RTR^{-1}$ and leverage the symmetry of $C$ to obtain
\[
(A_{ij})^\top
=(T^{\ell(i,j)}R)^\top
=R^{-1}T^{-\ell(i,j)}
=T^{\ell(i,j)}R^{-1}
=T^{\ell(j,i)}R
=A_{ji}.
\]
Finally, for (D4), we start by applying (D1)--(D3) to get
\[
(A^2)_{ii}
=\sum_{j=1}^n A_{ij}A_{ji}
=\sum_{\substack{j=1\\j\neq i}}^n A_{ij}(A_{ij})^{-1}
=(n-1)I_p.
\]
For $i\neq j$, we apply (D3) and the definition of $A$ to obtain
\begin{equation}
\label{eq.sum of translation operators}
(A^2)_{ij}
=\sum_{k=1}^n A_{ik}A_{kj}
=\sum_{k=1}^n A_{ik}(A_{jk})^\top
=\sum_{\substack{k=1\\k\not\in\{i,j\}}}^n T^{\ell(i,k)-\ell(j,k)}.
\end{equation}
Considering $C_{ik}\overline{C_{jk}}=\omega^{\ell(i,k)-\ell(j,k)}$, our proof of the first claim establishes that for each $\ell\in\{0,\ldots,p-1\}$, the matrix $T^\ell$ appears as a term of the right-hand side of \eqref{eq.sum of translation operators} exactly $a_\ell=(n-2)/p$ times.
Thus, $(A^2)_{ij}=\frac{n-2}{p}\cdot J_{p}$, and the claim follows.
\end{proof}

\section*{Acknowledgments}

JWI was supported by the Air Force Research Laboratory AFIT Directorate, through the Air Force Office of Scientific Research Summer Faculty Fellowship Program, Contract Numbers FA8750-15-3-6003, FA9550-15-0001 and FA9550-20-F-0005.
JJ was supported by NSF DMS 1830066.
DGM was supported by AFOSR FA9550-18-1-0107 and NSF DMS 1829955.
The views expressed in this article are those of the authors and do not reflect the official policy
or position of the United States Air Force, Department of Defense, or the U.S. Government.

\end{document}